\documentclass[reqno,12pt, oneside]{amsart}

\usepackage{amssymb}
\usepackage[hidelinks]{hyperref}
\usepackage[textsize=footnotesize]{todonotes}
\usepackage{tikz}
\usetikzlibrary{cd}
\usepackage{float}
\usepackage{xcolor}
\usepackage{todonotes}
\usepackage{framed}

\theoremstyle{plain}
\newtheorem{theorem}{Theorem}[section]
\newtheorem{lemma}[theorem]{Lemma}
\newtheorem{definition}[theorem]{Definition}
\newtheorem{corollary}[theorem]{Corollary}

\newtheorem{remark}[theorem]{Remark}

\DeclareMathOperator{\var}{var}

\newcommand{\R}{\mathbb R}

\newcommand{\N}{\mathbb N}

\newcommand{\I}{\mathcal I}
\newcommand{\J}{\mathcal J}
\newcommand{\cP}{\mathcal P}
\newcommand{\Fin}{\mathrm{Fin}}
\newcommand{\Exh}{\mathrm{Exh}}

\newcommand{\BV}{\mathrm{BV}}

\newcommand{\norm}[1]{\left\|#1\right\|}

\newcommand{\abs}[1]{\lvert#1\rvert}

\newcommand{\zdef}{{\mathrel{\mathop:}=}}

\newcommand{\ABV}{{ABV}}
\newcommand{\BBV}{{BBV}}

\title[BV Compactness from ideal perspective]{Compactness in spaces of functions of bounded variation from ideal perspective}

\author[J.~Gulgowski]{Jacek Gulgowski}
\address[J.~Gulgowski]{Institute of Mathematics\\ Faculty of Mathematics, Physics and Informatics\\ University of Gda\'{n}sk\\ ul. Wita Stwosza 57\\ 80-308 Gda\'{n}sk\\ Poland}
\email{Jacek.Gulgowski@ug.edu.pl}

\author[A.~Kwela]{Adam Kwela}
\address[A.~Kwela]{Institute of Mathematics\\ Faculty of Mathematics\\ Physics and Informatics\\ University of Gda\'{n}sk\\ ul.~Wita  Stwosza 57\\ 80-308 Gda\'{n}sk\\ Poland}
\email{Adam.Kwela@ug.edu.pl}
\urladdr{https://mat.ug.edu.pl/~akwela}

\author[J.~Tryba]{Jacek Tryba}
\address[J.~Tryba]{Institute of Mathematics\\ Faculty of Mathematics, Physics and Informatics\\ University of Gda\'{n}sk\\ ul. Wita Stwosza 57\\ 80-308 Gda\'{n}sk\\ Poland}
\email{Jacek.Tryba@ug.edu.pl}

\begin{document}

\begin{abstract}
Recently we have presented a unified approach to two classes of Banach spaces defined by means of variations (Waterman spaces and Chanturia classes), utilizing the concepts from the theory of ideals on the set of natural numbers. We defined correspondence between an ideal on the set of natural numbers, a certain sequence space and related space of functions of bounded variation. In this paper, following these ideas, we give characterizations of compact embeddings between different Waterman spaces and between different Chanturia classes: both in terms of sequences defining these function spaces and in terms of properties of ideals corresponding to these function spaces.
\end{abstract}




\maketitle

\section{Introduction}

In this paper we are going to answer some questions about compact embeddings between certain spaces of functions of bounded variation. 

Until recently, the concept of compactness in the context of spaces of functions of bounded variation was not deeply understood. Formally, the characterization of compactness of subsets in the space $BV$ of functions of bounded Jordan variation is known since decades: it may be found in the Part I of the celebrated Dunford and Schwarz monograph 
\cite[Exercise IV.13.48]{DS}. According to this Exercise, the space $BV$ may be decomposed (as a direct sum) into its subspace $NBV$ (consisting of functions $g$ that are right-continuous  and such that $g(0)=0$) and the subspace consisting of all functions that vanish except for a countable set of points (the later is isomorphic to $L^1$). By \cite[Exercise IV.13.34]{DS}, a set $K\subseteq NBV$ is relatively compact if and only if there exists a function $g\in NBV$ such that the map
\[
A_{\mu_g} \ni h \mapsto \int_{[0,1]} h(s)\mu_f(\textup ds)
\]
is uniformly continuous with respect to $f\in K$ (here $\mu_g$ is the measure corresponding to $g\in NBV$ and $A_{\mu_g}$ is the set of all functions $x\in NBV$ such that $|x(t)|\leq 1$ for almost every $t$ with respect to the measure $\mu_g$). As we can see, this characterization is very complex and seems to be quite far from being practical and useful. It also does not look like a good starting point for generalizations to spaces of functions of bounded variation for different generalizations of Jordan variation. There were some attempts to characterize compactness in the space of continuous functions of bounded variation (see e.g. \cite{PW} and \cite{C}). Recently, a characterization of compactness of completely different nature was given (see \cite{BG20}, \cite{G23}, \cite{GKM} and \cite{SiXu}). 

So far compact embeddings between different spaces of functions of bounded variation have been considered in different contexts: in \cite{CO} by Ciemnoczo\l{}owski and Orlicz  (for functions of bounded $p$-variation and Young variation) and in \cite{BCGS} and \cite{BGK2} (for Waterman and Young variation spaces).

In this paper, we give characterizations of compact embeddings between different Waterman spaces and between different Chanturia classes. We will utilize our recent ideas from \cite{GKT}, where we have presented a unified approach to Waterman spaces and Chanturia classes using the theory of ideals on the set of natural numbers. Both of our characterizations are formulated in terms of sequences defining these function spaces as well as in terms of properties of ideals corresponding to these function spaces.

\section{Preliminaries}

\subsection{Basics about Waterman spaces}

Let us assume that $A=(a_n)_{n\in\N}$ is such a  nonincreasing sequence of positive real numbers that $\sum_{n=1}^{\infty} a_n = +\infty$. We call such sequence \emph{a Waterman sequence}. If additionally $\lim_{n\to +\infty}a_n = 0$, we say that the sequence $A$ is a \emph{proper Waterman sequence}.

Let us denote the unit interval by $I=[0,1]$. Moreover, by ${\mathcal P}_I$ we denote the set of all infinite sequences of nonoverlapping, closed  subintervals $\{I_1, I_2, \ldots, I_N, \ldots\}$ of $I$. The intervals may be {\em degenerate}, i.e. it may happen that $I_n$ consists of only one point. 

 \begin{definition}
  Let $A = (a_n)_{n\in\mathbb{N}}$ be a Waterman sequence and let $x\colon I\to\mathbb{R}$. We say that $x$ is of bounded $A$-variation  if there exists a positive  constant $M$ such that for any sequence of nonoverlapping  subintervals $\{I_1, I_2, \ldots, I_N, \ldots\} \in {\mathcal P}_I$, the following inequality holds
  \[
   \sum_{n=1}^{+\infty} a_n\abs{x(I_n)} \leq M,
  \]
  where $I_n = [s_n,t_n]$ and $|x(I_n)| = |x(t_n)-x(s_n)|$.
  The supremum of the above sums, taken over the family ${\mathcal P}_I$ of all sequences of nonoverlapping closed subintervals of $I$, is called the $A$-variation of $x$  and it is denoted by $\var_A(x)$. 
 \end{definition}

 \begin{remark}
 The special case of a sequence constantly equal to $1$ corresponds to the classical Jordan variation of a function $x$, which we will denote by $\var(x)$.
 \end{remark}
 
The concept of $A$-variation was introduced by Waterman in \cite{Wat1}. Since then functions of bounded $A$-variation were intensively studied by many authors -- for an overview we refer to \cite{ABM} and the newer position \cite{ABKR}.
 
It is worth to mention that there are many equivalent ways to express that the function $x\colon I\to\R$ is of bounded $A$-variation (cf. \cite[Theorem 1, p.~34]{Wat1976} and \cite[Proposition 1]{BCGS}), but we will not recall those results here.

The space of all  functions defined on the interval $I$ and  of bounded $A$-variation, endowed with the norm $\|x\|_{\ABV} \zdef \abs{x(0)} + \var_A(x)$ forms a Banach space $\ABV(I)$ (see \cite[Section 3]{Wat1976}).

The spaces $\ABV(I)$ are proper subspaces of the space $B(I)$ of all bounded functions $x\colon I\to\R$. The space $B(I)$ is equipped with the standard supremum norm
\[
\norm{x}_\infty = \sup_{t\in I} |x(t)|.
\]

\subsection{Chanturia classes}\label{sec:Chanturia}

In 1974 Chanturia introduced the concept of the {\em modulus of variation of the bounded function} (see \cite{Cha74}), which for $x\colon I\to\R$ is given as the sequence 
\[
v(x,n) = \sup_{\{I_1,\ldots,I_n\}\in{\mathcal P}_n} \sum_{k=1}^n |x(I_k)|,
\]
where ${\mathcal P}_n$ denotes the set of all $n$-element sequences of nonoverlapping closed subintervals of $I$. In the mentioned paper \cite{Cha74}, for a given sequence $g\colon \N\to \R$, the author introduced $V[g(n)]$ as the family of those functions $x\colon I\to\R$, for which $v(x,n) = O(g(n))$, i.e. there is $\eta>0$ such that
\[
\frac{v(x,n)}{g(n)}\leq\eta
\]
for all $n\in\N$. These classes are now called {\em Chanturia (or Chanturiya) classes} in literature. One of the results (see \cite[Theorem 1]{Cha74}) was that the necessary and sufficient condition for a sequence to be $v(x,n)$, for some function $x$, is that it is nondecreasing and concave.   These classes were studied since then in many papers, mainly in relation to the  convergence of Fourier series and relations to other families of functions of bounded variation (see especially the relation between Chanturia classes and Waterman spaces given by Avdispahi\'c in \cite{Avdispahic}).


\subsection{Basics about ideals}

 \begin{definition}
 A family $\I\subseteq\cP(\N)$ is called an \emph{ideal} whenever
 \begin{itemize}
    \item $\N\notin\I$,
    \item if $F\subseteq\N$ is finite, then $F\in\I$,
    \item if $C\in\I$ and $D\subseteq C$, then $D\in\I$,
    \item if $C,D\in\I$, then $C\cup D\in\I$.
 \end{itemize}
 \end{definition}

 By $\Fin$ we denote the smallest ideal, i.e. the one consisting only of all finite subsets of $\N$. 

   \begin{definition}
If $\I$ and $\J$ are ideals (or $\I$ is an ideal and $\J=\mathcal{P}(\N)$), then we say that \emph{$\I$ is below $\J$ in the Kat\v{e}tov order} and write $\I\leq_K\J$ whenever there is a function $f:\N\to\N$ such that $f^{-1}[C]\in\J$ for every $C\in\I$.
 \end{definition}

Kat\v{e}tov order was introduced in the 1970s in papers \cite{Kat2} and \cite{Kat1} by Kat\v{e}tov. Note that actually, despite its name, Kat\v{e}tov order is only a pre-order, not a partial order (it is not antisymetric). 

An ideal $\I$ is a \emph{P-ideal} if for every sequence $(A_n)_{n\in\N}$ of elements of $\I$ there is $A\in\I$ such that $A_n\setminus A$ is finite for all $n\in\N$. We say that an ideal $\I$ is \emph{tall} if for every infinite $C\subseteq\N$ there is an infinite $D\subseteq C$ such that $D\in\I$. It is easy to see that $\I$ is not tall if and only if  $\I\leq_K\Fin$. Consequently, all non-tall ideals are $\leq_K$-equivalent (i.e., $\I\leq_K\J$ and $\J\leq_K\I$ for any two non-tall ideals $\I$ and $\J$).

By identifying subsets of $\N$ with their characteristic functions, we can treat ideals as subsets of the Cantor space $\{0,1\}^\N$ and therefore assign topological properties (such as being $\bf{F_\sigma}$ or analytic) to ideals on $\N$.

\subsection{Submeasures -- correspodence between ideals and variations}

A function $\phi:\mathcal{P}(\N)\to[0,\infty]$ is called a \emph{submeasure} if $\phi(\emptyset)=0$, $\phi(\{n\})<\infty$ for every $n\in\N$, and 
\[
\phi(C)\leq\phi(C\cup D)\leq\phi(C)+\phi(D)
\]
for all $C,D\subseteq\N$. A submeasure $\phi$ is \emph{lower semicontinuous} (lsc, in short) if $\phi(C)=\lim_{n\to\infty}\phi(C\cap\{1,2,\ldots,n\})$ for each $C\subseteq\N$.

Solecki in \cite[Theorem 3.1]{SoleckiExh} showed that an ideal is an analytic P-ideal if and only if it is of the form
$$\Exh(\phi)=\left\{C\subseteq\N:\ \lim_{n\to\infty}\phi(C\setminus \{1,2,\ldots,n\})=0\right\}$$
for some lower semicontinuous submeasure $\phi$ such that $\N\notin\Exh(\phi)$ (see also \cite[Theorem 1.2.5]{Farah}). 

For examples of ideals induced by submeasures see the next two Sections and \cite[Example 1.2.3]{Farah}. 

Given an lsc submeasure $\phi$, let $\mathcal{M}_\phi$ be the family of all measures $\mu$ on $\N$ such that $\mu\leq\phi$. Then, by definition, $\phi$ is \emph{non-pathological}, if $\phi(C)=\sup\{\mu(C): \mu\in\mathcal{M}_\phi\}$ for all $C\subseteq\N$. Not every lsc submeasure is non-pathological -- see \cite[Section~1.9]{farah-book}, \cite[Theorem~4.12]{ft-mazur}, \cite[Section~6.2]{marciszewski-sobota}, \cite{MezaPat} or \cite[Theorem~4.7]{tryba-gendensity} for such examples. 

The two definitions given below are basically repeated after sections 3.3 and 3.4 of \cite{GKT}. We refer the  Reader to these sections for further details, examples and discussions.

\begin{definition}
Let $\phi$ be a non-pathological lsc submeasure. Define a function $\hat\phi:\R^\N\to[0,\infty]$ by
\[
\hat\phi(x)=\sup\left\{\sum_{n\in\N}\mu(\{n\})|x_n|: \mu\in\mathcal{M}_\phi\right\}
\]
for all $x=(x_n)_{n\in\N}\in\R^\N$. 
\end{definition}

The above concept was introduced by Borodulin-Nadzieja and Farkas in \cite[Proposition 5.3]{FBN}.

\begin{definition}
Let $\phi$ be a non-pathological lsc submeasure. For $J=(J_n)_{n\in\N}\in\mathcal{P}_I$ denote by $x(J)$ the sequence $(|x(J_n)|)_{n\in\N}$. Define
\[
\BV(\phi)=\left\{x\in B(I): \sup_{J\in\mathcal{P}_I}\hat\phi(x(J))<\infty\right\}.
\]
\end{definition}

As proved in \cite{GKT}, this is a Banach space normed by:
\[
\|x\|_\phi=|x(0)|+\sup_{J\in\mathcal{P}_I}\hat\phi(x(J)).
\]

\section{Waterman spaces}

\subsection{Submeasures and ideals associated to Waterman spaces}

For a sequence of positive real numbers $A=(a_n)_{n\in\N}$ denote
$$\I_A=\left\{C\subseteq\N: \sum_{n\in C}a_n<\infty\right\}.$$
Note that either $\I_A=\mathcal{P}(\N)$ (if $\sum_{n=1}^\infty a_n<\infty$) or $\I_A$ is an ideal (if $\sum_{n=1}^\infty a_n=\infty$).

  \begin{definition}
  \label{def-summable}
An ideal is called a \emph{summable ideal} if it is of the form 
$\I_A$, for some sequence of positive real numbers $A=(a_n)_{n\in\N}$ such that $\sum_{n=1}^\infty a_n=\infty$.
 \end{definition}

\begin{remark}
Note that in the definition of summable ideals we do not require that $A$ is nonincreasing. However, in our paper we will only consider summable ideals given by nonincreasing sequences, that is, by Waterman sequences.
 \end{remark}

Note that $\Fin$ is a summable ideal (given by the sequence constantly equal to $1$). If $A=(a_n)_{n\in\N}$
 is a sequence of positive real numbers such that $\sum_{n=1}^\infty a_n=\infty$, then $\I_A$ is tall if and only if $\lim_{n\to\infty}a_n=0$.

For any sequence of positive real numbers $A=(a_n)_{n\in\N}$ such that $\sum_{n=1}^\infty a_n=\infty$ we can define an lsc submeasure $\phi_A$ by the formula $\phi_A(C)=\sum_{n\in C}a_n$. Observe that in this case $\I_A=\Exh(\phi_A)$ and $\BV(\phi_A)=\ABV$.

In \cite{GKT} we investigated connections between summable ideals and $\ABV$ spaces. In particular, we have shown the following.

\begin{theorem}[{\cite[Theorem 6.2]{GKT}}]\label{th:inclusions}
Assume we have two  Waterman sequences $A=(a_n)_{n\in\N}$ and $B=(b_n)_{n\in\N}$. The following statements are equivalent:
\begin{itemize}
\item[(a)] $\sum_{i=1}^n b_i = O(\sum_{i=1}^n a_i)$, i.e., there is $\eta>0$ such that 
\[
\frac{\sum_{i=1}^n b_i}{\sum_{i=1}^n a_i}\leq\eta
\]
for all $n\in N$,
\item[(b)] $\I_A\leq_K\I_B$,
\item[(c)] $\ABV\subseteq\BBV$.
\end{itemize}
\end{theorem}

Equivalence of items (a) and (c) from the above result was proved earlier by Perlman and Waterman (see \cite[Theorem 3]{PerlmanWaterman}). We will also need the following simple observation.

\begin{remark}
\label{rem:inclusions}
Recall that the definition of $\BBV$ requires that $\sum_{i=1}^\infty b_i=\infty$. However, if $A=(a_n)_{n\in\N}$ is a Waterman sequence, but $B=(b_n)_{n\in\N}$ is any sequence of positive real numbers such that the series $\sum_{i=1}^\infty b_i$ is convergent, then equivalence of items (a) and (b) from Theorem \ref{th:inclusions} still holds true (since $\sum_{i=1}^n b_i = O(\sum_{i=1}^n a_i)$ and $\I_A\leq_K\mathcal{P}(\N)=\I_B$). Moreover, even if  $\BBV$ is not a Waterman space in this case, we get $BV(\phi_B)=B(I)$, so the inclusion given in (c), i.e. $BV(\phi_A)\subseteq BV(\phi_B)$, holds true as well.
\end{remark}

\subsection{Compact embeddings of Waterman spaces}

In this paper we are going to investigate when the embedding $\ABV\subseteq\BBV$ is compact (i.e. making bounded subsets of $\ABV$ relatively compact in $\BBV$). The quick and simple answer is that this embedding is never compact: the set 
\[
K = \{ \chi_{[0,t]} : t\in [0,1) \} \subseteq \BV(I)\subseteq \ABV
\]
is a set which is bounded in all $\ABV$ spaces and not relatively compact in any of them (nor in $B(I)$).

But the problem of compact embedding may be investigated more carefully. 
Note that given two Waterman sequences $A$ and $B$ such that $\ABV\subseteq\BBV\subseteq B(I)$ (where the inclusions are actually continuous embeddings), every bounded $K\subseteq\ABV$, which is relatively compact in $\BBV$, has to be relatively compact in the $\Vert\cdot\Vert_\infty$ norm.  This is an easy observation implied by the fact that the norm in $\ABV$ is stronger than both the supremum norm $\Vert\cdot\Vert_\infty$ and the norm in $\BBV$. Hence, the necessary condition for the compactness in $\BBV$ is the compactness in $\Vert\cdot\Vert_\infty$ norm. A result in this direction was given  in \cite[Proposition 5]{BGK2} where, instead of assuming that set is relatively compact in the supremum norm, authors assumed that the functions in a set are equicontinuous.  

The results mentioned above may be easily generalized, actually by repeating the ideas of the proof given in the mentioned paper.
\begin{theorem}\label{th:compact}
Let $A$ and $B$ be two Waterman sequences such that $B=o(A)$, i.e. $\lim_{n\to+\infty} \frac{b_n}{a_n} = 0$.  Let $K\subseteq \ABV$ be a bounded subset of $\ABV$, which is relatively compact in the  supremum norm. Then $K\subseteq\BBV$ is relatively compact.
\end{theorem}
\begin{proof}Let us take any sequence $(x_n)_{n\in\mathbb{N}}\subseteq K \subseteq \ABV$. Since $K$ is relatively compact in the supremum norm, taking an appropriate subsequence we may assume that $\norm{x_n-x}_\infty\to 0$ for some bounded function $x\in B(I)$. We are going to show that $x\in \BBV$ and that $\var_B(x_n-x)\to 0$.

We may obviously take such  positive constant $C>0$ that $b_n\leq C a_n$ for all $n\in\N$. This implies that
\[
\var_{B}(y) \leq C \var_A(y)
\]
for any function $y\in\ABV$. Since the sequence $(x_n)_{n\in\mathbb{N}}$ is bounded in $\ABV$, we may assume that there exists such constant $M>0$ that $\var_A(x_n)\leq M$ and by taking the appropriate subsequence (by a version of Helly selection theorem, cf. \cite[Theorem 5]{Wat1976})  we can say that $\var_A(x)\leq M$ and 
\[
\var_B(x) \leq C\var_A(x) \leq CM.
\] 
Now we know that $x_n$ converges uniformly to $x$ and that $x\in\BBV$. What remains is the proof that $\var_B(x_n-x)\to 0$. We may basically repeat the estimates from the proof of \cite[Proposition 5]{BGK2}: let us take any $\varepsilon>0$ and such $j\in \N$ that $b_i \leq \frac{\varepsilon}{4M} a_i$ for all $i\geq j$. Since $\norm{x_n-x}_\infty\to 0$, we may assume that there exists $k_0\in\N$ such that  for all intervals $J\subseteq I$ and all $k\geq k_0$ we have  
\[
|(x_k-x)(J)| \leq 2\norm{x_k-x}_\infty < \frac{\varepsilon}{2\sum_{i=1}^{j} b_i}.
\]
Let us now take any $k>k_0$ and any sequence $(I_i)_{i\in\N}$ of nonoverlapping closed intervals and estimate
\[
\sum_{i\in \N} b_i|(x_k-x)(I_i)| = \sum_{i=1}^j  b_i|(x_k-x)(I_i)| + \sum_{i=j+1}^{+\infty} b_i|(x_k-x)(I_i)| \leq 
\]
\[
\frac{\varepsilon}{2} + \frac{\varepsilon}{4M} \sum_{i=j+1}^{+\infty} a_i|(x_k-x)(I_i)| \leq \frac{\varepsilon}{2} + \frac{\varepsilon}{4M}2M = \varepsilon,
\]
which completes the proof.
\end{proof}

In addition to the above, in \cite{BCGS} the Authors tried to reverse this result: the question was if for any compact $K\subseteq \BBV$, there exists a subspace $\ABV\subsetneq \BBV$ such that $K\subseteq \ABV$. The answer was given for a certain class of  $\ABV$ spaces (see \cite[Theorem 18]{BCGS}). 

In this paper we are going to extend Theorem \ref{th:compact} in a way that  will give a complete characterization of sequences $A$ and $B$ satisfying the modified compact embedding property, i.e. such that bounded subsets of $\ABV$, which are relatively compact in the supremum norm, are relatively compact in $\BBV$.

\begin{lemma}\label{lem:comp:emb2}Assume we have two Waterman sequences $A=(a_n)_{n\in\N}$ and $B=(b_n)_{n\in\N}$ and that $B$ is proper. The following conditions are equivalent:
\begin{enumerate}
\item[(a)] $\sum_{k=1}^n b_k = o(\sum_{k=1}^n a_k)$, i.e.
\[
\lim_{n\to +\infty} \frac{\sum_{k=1}^n b_k}{\sum_{k=1}^n a_k}=0,
\]
\item[(b)] for any $\varepsilon>0$ there exists $k_0\in\N$ such that for all nonincreasing, converging to $0$ sequences $(c_n)_{n\in N}\subseteq[0,+\infty)$ such that $\sum_{n=1}^{\infty} a_nc_n < +\infty$ the inequality holds
\[
\sum_{n=k_0+1}^{+\infty} b_nc_n < \varepsilon \sum_{n=1}^{+\infty} a_n c_n. 
\]
\end{enumerate}
\end{lemma}
\begin{proof}
Let us first prove the implication (b) $\Rightarrow$ (a). Let us fix $\varepsilon>0$ and take $k_0=k_0(\frac{\varepsilon}{2})\in\N$ as in (b). Now, because $\sum_{n=1}^{\infty} a_n = +\infty$, there exists such $k_1\geq k_0$ that
\[
\sum_{k=1}^{k_0} b_k < \frac{\varepsilon}{2} \sum_{k=1}^{k_1} a_k.
\]
Then for $n\geq k_1$, if $c_k = 1$ for $k=1,...,n$ and $c_k=0$ for $k=n+1, n+2, ...$ (note that the sequence $(c_k)_{k\in\N}$ may be used in the condition (b)), we get
\[
\sum_{k=1}^{n} b_k = \sum_{k=1}^{k_0} b_k + \sum_{k=k_0+1}^n b_k \leq \frac{\varepsilon}{2} \sum_{k=1}^{k_1} a_k + \sum_{k=k_0+1}^{\infty} b_kc_k \leq
\]
\[
\frac{\varepsilon}{2} \sum_{k=1}^{k_1} a_k + \frac{\varepsilon}{2} \sum_{k=1}^{\infty} a_kc_k =  \varepsilon \sum_{k=1}^{k_1} a_k+ \frac{\varepsilon}{2} \sum_{k=k_1}^{n} a_k  \leq  \varepsilon \sum_{k=1}^{n} a_k.
\]
This completes the proof of the first implication.

Let us now proceed to the proof of the implication (a)$\Rightarrow$ (b). Let us fix $\varepsilon>0$. There exists such $k_0\in\N$  that for $n\geq k_0$
\[
\sum_{k=1}^n b_k \leq \varepsilon \sum_{k=1}^n a_k.
\]
Let us now take any nonincreasing sequence $(c_n)_{n\in\N}$ which converges to $0$ and satisfies $\sum_{n=1}^{\infty} a_nc_n <+\infty$ and any $n>k_0$. Then
\[
\sum_{k=k_0+1}^{n} b_kc_k \leq c_{k_0}\sum_{k=1}^{k_0} b_k +  \sum_{k=k_0+1}^{n} b_kc_k = 
\]
\[
c_{k_0}\sum_{k=1}^{k_0} b_k - c_{k_0+1}\sum_{k=1}^{k_0} b_k + c_{k_0+1}\sum_{k=1}^{k_0+1} b_k - c_{k_0+2}\sum_{k=1}^{k_0+1} b_k + c_{k_0+2}\sum_{k=1}^{k_0+2} b_k - ....
\]
\[
...-c_n\sum_{k=1}^{n-1}b_k + c_n\sum_{k=1}^n b_k = 
\]
\[
\sum_{j=k_0}^{n-1} (c_j-c_{j+1})\sum_{k=1}^j b_k + c_n\sum_{k=1}^n b_k \leq
\]
\[
\varepsilon \sum_{j=k_0}^{n-1} (c_j-c_{j+1})\sum_{k=1}^j a_k +  \varepsilon c_n\sum_{k=1}^n a_k =
\]
\[
\varepsilon \left( c_{k_0}\sum_{k=1}^{k_0} a_k + \sum_{k=k_0+1}^n c_ka_k \right)  = \varepsilon \left(\sum_{k=1}^{k_0}  c_{k_0}a_k + \sum_{k=k_0+1}^n c_ka_k \right)  \leq
\]
\[
\varepsilon \left(\sum_{k=1}^{k_0}  c_{k}a_k + \sum_{k=k_0+1}^n c_ka_k \right) = \varepsilon\sum_{k=1}^n c_ka_k.
\]
This means that condition (b) is satisfied and it completes the proof.
\end{proof}

\begin{lemma}\label{lem:comp:emb}
Assume we have two Waterman sequences $A=(a_n)_{n\in\N}$ and $B=(b_n)_{n\in\N}$ and that $B$ is proper. The following statements are equivalent:
\begin{enumerate}
\item[(a)] if $K\subseteq \ABV$ is bounded in $\ABV$ and relatively compact in the $\Vert\cdot\Vert_\infty$ norm then $K\subseteq\BBV$ and $K$ is relatively compact in the space $\BBV$,
\item[(b)] for any $\varepsilon>0$ there exists $k_0\in\N$ such that for all nonincreasing converging to $0$ sequences $(c_n)_{n\in N}\subset[0,+\infty)$ such that $\sum_{n=1}^{\infty} a_nc_n < +\infty$ the inequality holds
\[ 
\sum_{n=k_0+1}^{+\infty} b_nc_n < \varepsilon \sum_{n=1}^{+\infty} a_n c_n.
\]
\end{enumerate}
\end{lemma}
\begin{proof}
Let us start with the implication (b) $\Rightarrow$ (a). By Lemma \ref{lem:comp:emb2} and Theorem \ref{th:inclusions} we see that $\ABV\subseteq\BBV$. Let us take a bounded set $K\subseteq \ABV$ which is relatively compact in the supremum norm. Let $M>0$ be such that for all $x\in K$ we have $\norm{x}_\ABV\leq M$ and let $(x_n)_{n\in\N}\subseteq K$. By Helly selection theorem (cf. \cite[Theorem 5]{Wat1976}), we can take a subsequence $(x_{n_k})_{k\in\N}$ pointwise convergent to some function $x\in \ABV$ such that $\norm{x}_\ABV\leq M$. Because $K$ is relatively compact in the uniform convergence norm, we can also assume that $(x_{n_k})_{k\in\N}$ converges to $x$ uniformly. 

Now we are going to show that the subsequence $(x_{n_k}-x)_{k\in\N}$ converges to $0$ in the norm of the space $\BBV$. Let us take $\varepsilon>0$. Of course it is enough to show that $\var_B(x_{n_k}-x)\leq\varepsilon$ for $k\in\N$ big enough. For $\varepsilon / (4M)$ there exists $k_0\in\N$ as in condition (b). For $k_0$ there exists such $k_1\in\N$  that
\[
\norm{x_{n_k}-x}_\infty \leq \frac{\varepsilon}{4\sum_{i=1}^{k_0} b_i}
\]
for $k>k_1$.

In order to estimate the variation $\var_B(x_{n_k}-x)$, let us take any sequence $(I_n)_{n\in\N}$ of nonoverlapping closed intervals and fix $k>k_1$. We may assume that the sequence
\[
c_n = |(x_{n_k}-x)(I_n)|
\]
is nonincreasing. We can see that $\sum_{n=1}^{\infty} a_nc_n \leq 2M$ and that $c_n\to 0$. 

Let us estimate
\[
\sum_{n=1}^{\infty} b_nc_n = \sum_{n=1}^{k_0} b_nc_n +  \sum_{n=k_0+1}^{\infty} b_nc_n \leq
\]
\[
2\norm{x_{n_k}-x}_\infty\left( \sum_{n=1}^{k_0} b_n\right) + \frac{\varepsilon}{4M} \sum_{n=1}^{+\infty} a_n c_n \leq \frac{\varepsilon}{2} + \frac{\varepsilon}{2} = \varepsilon.
\]

Let us now proceed to the proof of the implication (a)$\Rightarrow$ (b). Contrary to our claim, assume that (a) holds true and that there exists $\varepsilon_0>0$ such that for all $k\in\N$ there exists such a nonincreasing and converging to $0$ sequence $(c_n^{(k)})_{n\in\N}$ that $\sum_{n=1}^{\infty} a_n c_n^{(k)}<+\infty$ and
\begin{equation}\label{cond:not:anbncn}
\sum_{n=k+1}^{+\infty} b_nc_n^{(k)} \geq \varepsilon_0  \sum_{n=1}^{\infty} a_n c_n^{(k)}.
\end{equation}
In the condition \eqref{cond:not:anbncn} we may decrease the right-hand side by taking $c_n^{(k)} = c_k^{(k)}$ for all $n=1,...,k-1$. So we can say that the condition holds true for a sequence $(c_n^{(k)})_{n\in\N}$ having the first $k$ terms constant.
As we can see the condition \eqref{cond:not:anbncn} is homogeneous with respect to the sequence $(c_n^{(k)})_{n\in\N}$, so we can assume that $\sum_{n=1}^{\infty} a_n c_n^{(k)} = 1$.

Let us observe that there must be 
\[
\lim_{k\to+\infty} \left( \sup_{n\in\N} |c_n^{(k)}|\right) = \lim_{k\to +\infty} c_k^{(k)} = 0.
\]
This is because $ \sum_{n=1}^k a_n\to\infty$ and 
\[
1 = \sum_{n=1}^\infty a_n c_n^{(k)} \geq \sum_{n=1}^{k} a_n c_n^{(k)} = c_k^{(k)} \sum_{n=1}^k a_n. 
\]

Let us now take the sequence of dyadic intervals $I_n = [\frac{1}{2^n},\frac{1}{2^{n-1}}]$ and define functions $x_k\colon I\to\R$ as piecewise linear functions such that $x_k(1) = 0$, 
$x_k(\frac{1}{2}) = c_1^{(k)}$, $x_k(\frac{1}{4}) =  c_1^{(k)} - c_2^{(k)}$ and in general
\[
x_k\left(\frac{1}{2^{n}}\right) = \sum_{i=1}^n (-1)^{i+1}c_i^{(k)}.
\]

As we can see using the Leibniz criterion, the series $\sum_{i=1}^{+\infty} (-1)^{i+1}c_i^{(k)}$ converges. Hence, when we set $x_k(0) = \sum_{i=1}^{+\infty} (-1)^{i+1}c_i^{(k)}$, we can see that the function $x_k$ is actually continuous. Moreover, $|x_k(I_n)| = c_n^{(k)}$ for all $n,k\in\N$. Consequently, $\norm{x_k}_\ABV=\sum_{n=1}^{\infty} a_n c_n^{(k)} = 1$ for each $k\in\N$, so $K=\{x_k:k\in\N\}$ is bounded in $\ABV$.

We can also observe that 
\[
\sup_{t\in I} |x_k(t)| = c_1^{(k)} = c_k^{(k)} 
\]
and tends to $0$ as $k\to +\infty$, so the sequence $(x_k)_{k\in\N}$ converges uniformly to $0$, thus $K$ is relatively compact in the supremum norm. If $K\not\subseteq \BBV$ then we immediately obtain the contradiction with (a), so we may assume otherwise.  We are now going to show that $(x_k)\subseteq \BBV$ does not converge to $0$ (so neither to any other function) in $\norm{\cdot}_\BBV$.

Let us estimate
\[
\var_B(x_k - 0) \geq \sum_{n=1}^{+\infty} |x_k(I_n)|b_n = \sum_{n=1}^{+\infty} c_n^{(k)}b_n \geq 
\]\[
\geq\sum_{n=k+1}^{+\infty} c_n^{(k)}b_n \geq\varepsilon_0  \sum_{n=1}^{\infty} a_n c_n^{(k)} = \varepsilon_0.
\]
This means that none of the subsequences of $(x_k)$ converges to $0$ in $\BBV$. But if any subsequence converges to some $x\in\BBV$, then there must be $x=0$. This means that the set $K=\{x_k : k\in\N\}$ contradicts item (a). This contradiction completes the proof.
\end{proof}

\begin{definition}
If $A=(a_n)_{n\in\N}$ is a sequence of positive real numbers and $M=(m_k)_{k\in\N}$ is an increasing sequence of natural numbers, then we define a new sequence $A^M=(a^M_n)_{n\in\N}$ of positive real numbers by $a_n^M=k a_n$, where $k\in\N$ is given by $m_k\leq n<m_{k+1}$ (and we put $a_n^M=a_n$ for all $n<m_1$).
\end{definition}

\begin{lemma}\label{lem:comp:emb:ideals}
Assume we have two Waterman sequences $A=(a_n)_{n\in\N}$ and $B=(b_n)_{n\in\N}$ and that $B$ is proper. The following statements are equivalent:
\begin{enumerate}
\item[(a)] $\I_A\leq_K\I_{B^M}$ for some increasing sequence $M=(m_k)_{k\in\N}$ of natural numbers,
\item[(b)] $\sum_{k=1}^n b_k = o(\sum_{k=1}^n a_k)$, i.e.
\[
\lim_{n\to +\infty} \frac{\sum_{k=1}^n b_k}{\sum_{k=1}^n a_k}=0.
\]
\end{enumerate}
\end{lemma}

\begin{proof}
(a)$\implies$(b): If $A$ is not proper then $\lim_{n\to\infty}a_n\neq 0$ and it is easy to see that in this case $b_n=o(a_n)$ which implies item (b). 

Assume now that $A$ is proper. Observe that $\lim_{n\to\infty}b^M_n=0$ as otherwise $\I_{B^M}$ would not be tall and consequently $\I_A$ would not be tall, which contradicts the fact that $A$ is a proper Waterman sequence.

Define $c_n=b^M_{m_1}$ for all $n<m_1$ and $c_n=\min\{b^M_i:m_1\leq i\leq n\}$ for all $n\geq m_1$. Then $\lim_{n\to\infty}c_n=0$ and $\I_A\leq_K\I_{C}$ (as $\I_A\leq_K\I_{B^M}\subseteq\I_C$, by the fact that $c_n\leq b_n^M$ for all $n\in\N$). Hence, either $\I_C=\mathcal{P}(\N)$ or $C=(c_n)_{n\in\N}$ is a proper Waterman sequence.

Fix $\varepsilon>0$ and let $k\in\N$ be such that $\frac{2\eta}{k}<\varepsilon$, where $\eta>0$ is such that 
\[
\frac{\sum_{i=1}^n c_i}{\sum_{i=1}^n a_i}<\eta
\]
for all $n\in\N$ (such $\eta$ exists by Theorem \ref{th:inclusions} and Remark \ref{rem:inclusions}). 

Observe that there is $j\in\N$ such that $c_n\geq k b_n$ for all $n\geq j$. Indeed, since $\lim_{n\to\infty}c_n=0$ and $C$ is nonincreasing, there is some $j\geq m_k$ such that $c_j=b^M_j\geq k b_j$. Then for each $n\geq j$ we have $c_n=b^M_l$ for some $l\in\N$ such that $n\geq l\geq j\geq m_k$, so $c_n=b^M_l\geq kb_l\geq k b_n$.

Let $n_\varepsilon>j$ be such that 
\[
\frac{k\sum_{i=1}^{j} b_i}{\sum_{i=1}^{n_\varepsilon} a_i}<\eta.
\]
Then for every $n\geq n_\varepsilon$ we have
\[
\frac{k\sum_{i=1}^{n} b_i}{\sum_{i=1}^{n} a_i}\leq\frac{k\sum_{i=1}^{j} b_i}{\sum_{i=1}^{n_\varepsilon} a_i}+\frac{\sum_{i=j+1}^{n} c_i}{\sum_{i=1}^{n} a_i}<2\eta.
\]
Thus, 
\[
\frac{\sum_{i=1}^{n} b_i}{\sum_{i=1}^{n} a_i}<\frac{2\eta}{k}<\varepsilon.
\]

(b)$\implies$(a): From condition (b) we know that for every $k\in\N$ there is $n_k\in\N$ such that 
\[
\frac{\sum_{i=1}^n b_i}{\sum_{i=1}^n a_i}<\frac{1}{k}
\]
for all $n\geq n_k$. Define $m'_1=n_1$ and for each $k\in\N$ let $m'_{k+1}\in\N$ be such that:
\begin{itemize}
    \item[(i)] $m'_{k+1}\geq n_{k+1}$,
    \item[(ii)] $m'_{k+1}>m'_{k}$,
    \item[(iii)] $k b_{m'_{k+1}-1}<\frac{1}{k}$,
    \item[(iv)] $k b_{m'_{k+1}-1}<(k-1)b_{m'_k-1}$.
\end{itemize}
Such $m'_{k+1}$ exists as $\lim_{n\to\infty}b_n=0$.

Put $c_n=\min\{k b_n,(k-1)b_{m'_{k}-1}\}$, where $k\geq 2$ is given by $m'_k\leq n<m'_{k+1}$ (and $c_n=b_n$ for $n<m'_2$). We will show that $\I_A\leq_K\I_C$. This is obvious if $\sum_{i=1}^\infty c_i<\infty$ (as in this case $\I_C=\mathcal{P}(\N)$), so assume that $\sum_{i=1}^\infty c_i=\infty$. 

Observe that $\lim_{n\to\infty}b^{M'}_n=0$ (by condition (iii)). Thus, thanks to condition (iv), $C=(c_n)_{n\in\N}$ is a proper Waterman sequence and we can use Theorem \ref{th:inclusions}. Define 
\[
\eta=\max\left\{1, \max\left\{\frac{\sum_{i=1}^n c_i}{\sum_{i=1}^n a_i} : n<m'_1\right\} \right\}.
\]
Let $n\in\N$. If $n<m'_1$ then clearly 
$$\frac{\sum_{i=1}^n c_i}{\sum_{i=1}^n a_i}\leq\eta.$$
 If $n\geq m_1$ then find $k\in\N$ such that $m'_k\leq n<m'_{k+1}$ and observe that
\[
\frac{\sum_{i=1}^n c_i}{\sum_{i=1}^n a_i}\leq \frac{k\sum_{i=1}^n b_i}{\sum_{i=1}^n a_i}<1\leq\eta
\]
(by condition (i)). Thus, $\I_A\leq_K\I_C$ by Theorem \ref{th:inclusions}.

Define now $m_k=m'_{k+1}$ for all $k\in\N$. Then $(m_k)_{k\in\N}$ is increasing (by item (ii)). Moreover, if $m_k=m'_{k+1}\leq i<m_{k+1}$ then $b^M_i=k b_i\leq \min\{(k+1) b_i,kb_{m'_{k+1}-1}\}=c_i$ (as $b_i\leq b_{m'_{k+1}-1}$ by the fact that $(b_i)_{i\in\N}$ is nonincreasing). Hence, $\I_C\subseteq\I_{B^M}$ and consequently $\I_A\leq_K\I_{B^M}$.
\end{proof}

\begin{theorem}
\label{thm:compactembeddings}
Assume that we have two Waterman sequences $A=(a_n)_{n\in\N}$ and $B=(b_n)_{n\in\N}$ and that $B$ is proper. The following statements are equivalent:
\begin{enumerate}
\item[(a)] if $K\subseteq \ABV$ is bounded in $\ABV$ and relatively compact in the $\Vert\cdot\Vert_\infty$ norm then $K\subseteq\BBV$ and $K$ is relatively compact in the space $\BBV$,
\item[(b)] $\I_A\leq_K\I_{B^M}$ for some increasing sequence $M=(m_k)_{k\in\N}$ of natural numbers,
\item[(c)] $\sum_{k=1}^n b_k = o(\sum_{k=1}^n a_k)$, i.e.
\[
\lim_{n\to +\infty} \frac{\sum_{k=1}^n b_k}{\sum_{k=1}^n a_k}=0.
\]
\end{enumerate}
\end{theorem}

\begin{proof}
Equivalence of items (a) and (c) follows from Lemmas \ref{lem:comp:emb} and \ref{lem:comp:emb2}. The equivalence of items (b) and (c) is proved in Lemma \ref{lem:comp:emb:ideals}.
\end{proof}

\begin{remark}
The condition $\sum_{k=1}^n b_k = o(\sum_{k=1}^n a_k)$ given in item (c) of Theorem \ref{thm:compactembeddings}
is more general than condition $b_n = o(a_n)$ considered before in Theorem \ref{th:compact} -- one can easily construct two Waterman sequences $A$ and $B$ such that $B$ is proper, $\sum_{k=1}^n b_k = o(\sum_{k=1}^n a_k)$, but $b_n = o(a_n)$ does not hold.
\end{remark}

\section{Chanturia classes}

\subsection{Submeasures and ideals associated to Chanturia classes}

\begin{definition}
An ideal is called a \emph{simple density ideal} if it is equal to $\Exh(\phi_g)$ for some lsc submeasure $\phi_g$ of the form
\[
\phi_g(C)=\sup_{n\in\N}\frac{|C\cap\{1,2,\ldots,n\}|}{g(n)},
\]
where $g:\N\to[0,\infty)$ is a function satisfying the following conditions:
\begin{itemize}
    \item[(a)] $(g(n))_{n\in\N}$ is nondecreasing;
    \item[(b)] $\lim_n g(n)=\infty$;
    \item[(c)] $(\frac{n}{g(n)})_{n\in\N}$ does not tend to zero.
\end{itemize}
 \end{definition}

\begin{remark}
\label{rem-simpledensity}
Note that for every function $g:\N\to[0,\infty)$ satisfying conditions (a)-(c) we have
\[
\Exh(\phi_g)=\left\{C\subseteq\N: \lim_{n\to\infty}\frac{|C\cap\{1,2,\ldots,n\}|}{g(n)}=0\right\}.
\]
Indeed, if $\lim_{n\to\infty}\frac{|C\cap\{1,2,\ldots,n\}|}{g(n)}=0$ and $\varepsilon>0$ is arbitrary, then there is $k_0\in\N$ such that $\frac{|C\cap\{1,2,\ldots,n\}|}{g(n)}\leq\varepsilon$ for all $n\geq k_0$, so also 
\[
\phi_g(C\setminus \{1,2,\ldots,k\})\leq\phi_g(C\setminus \{1,2,\ldots,k_0\})\leq\varepsilon,
\]
for every $k\geq k_0$, since for $n<k_0$ we get
\[
\frac{|(C\setminus \{1,2,\ldots,k_0\})\cap\{1,2,\ldots,n\}|}{g(n)}=0,
\]
and for $n\geq k_0$ we have 
\[
\frac{|(C\setminus \{1,2,\ldots,k_0\})\cap\{1,2,\ldots,n\}|}{g(n)}\leq\frac{|C\cap\{1,2,\ldots,n\}|}{g(n)}\leq\varepsilon.
\]
On the other hand, if $C\in\Exh(\phi_g)$ and $\varepsilon>0$ is arbitrary, then there is $k_0\in\N$ such that $\phi_g(C\setminus \{1,2,\ldots,k_0\})\leq\frac{\varepsilon}{2}$ and there is also $n_0\in\N$ such that $\frac{k_0}{g(n)}\leq\frac{\varepsilon}{2}$ for all $n\geq n_0$, so for $n\geq n_0$ we get
\[
\frac{|C\cap\{1,2,\ldots,n\}|}{g(n)}\leq \frac{k_0+|(C\setminus \{1,2,\ldots,k_0\})\cap\{1,2,\ldots,n\}|}{g(n)}\leq \varepsilon.
\]
\end{remark}

Simple denisty ideals have been extensively studied in \cite{simpleden3}, \cite{simpleden2} and \cite{simpleden1}. In \cite{GKT} we have studied spaces of the form $\BV(\phi_g)$ -- for some technical reasons, we needed one additional assumption:
\begin{itemize}
    \item[(d)] $(\frac{n}{g(n)})_{n\in\N}$ is nondecresing.
\end{itemize}
Note that (c) follows from (d). 

As the sequence  $(\frac{n}{g(n)})_{n\in\N}$ is nondecreasing, it has a limit: either finite or infinite. In case of a finite limit we can see that the space $\BV(\phi_g)$ actually equals to the space of all bounded functions and we will eliminate this case from our considerations. Hence, let us now assume that 
\begin{itemize}
\item[(e)] $\lim_{n\to+\infty} \frac{n}{g(n)} = +\infty$.
\end{itemize}

Assumptions (d) and (e) will be valid in the rest of this paper. We will denote by $\mathcal{G}$ the set of all functions $g:\N\to[0,\infty)$ satisfying conditions (a)-(e). 

As it is shown in \cite{GKT}, the space $\BV(\phi_g)$ corresponds to Chanturia class introduced in Section \ref{sec:Chanturia}, so we can write $V[g(n)] = BV(\phi_g)$. Moreover, in \cite{GKT} we have shown the following.

\begin{theorem}[{\cite[Theorem 5.1]{GKT}}]
\label{th:Chant:embed}
The following are equivalent for every $g,h\in\mathcal{G}$:
\begin{itemize}
\item[(a)] $g(n) = O(h(n))$, i.e., there is $\eta>0$ such that $\frac{g(n)}{h(n)}\leq\eta$ for all $n\in \N$,
\item[(b)] $\Exh(\phi_g)\subseteq\Exh(\phi_h)$,
\item[(c)] $V[g(n)]\subseteq V[h(n)]$.
\end{itemize}
\end{theorem}

Below we are going to show the conditions describing the situation, where the embedding $V[g(n)]\subseteq V[h(n)]$ is compact.

\subsection{Compact embeddings between Chanturia classes}

The classical Helly selection principle may be extended to Chanturia classes as the following theorem says (original formulation of \cite[Theorem 1]{Chistyakov} is slightly simplified here as we are considering real-valued functions as compared to functions with values in some metric space as originally considered by Chistyakov).

\begin{theorem}[{\cite[Theorem 1]{Chistyakov}}]
\label{thm:Helly:Chanturia}
Let $(x_j)_{j\in\N}$ be a sequence of bounded functions $x_j:I\to\R$ such that the sequence
\[
\mu(n) = \limsup_{j\to\infty} v(x_j, n)
\] 
satisfies $\mu(n) = o(n)$, i.e. $\lim_{n\to\infty} \frac{\mu(n)}{n} = 0$. Then there exists a subsequence $(x_{j_k})_{k\in\N}$ of $(x_j)_{j\in\N}$ pointwise convergent to some function $x:I\to\R$ satisfying $v(x,n) \leq \mu(n)$ for all $n\in \N$. 
\end{theorem}

This theorem leads to an immediate corollary:

\begin{corollary}
\label{cor:Helly:Chanturia}
Any sequence $(x_j)_{j\in\N}$ of functions from $V[g(n)]$, which is bounded by some $M>0$, has a subsequence $(x_{j_k})_{k\in\N}$ pointwise convergent to some function $x\in V[g(n)]$  such that $v(x,n) \leq M g(n).$
\end{corollary}

\begin{proof}
Since $(x_j)_{j\in\N}$ is a sequence of functions from $V[g(n)]$ bounded by $M$, we have
\[
v(x_j,n) \leq M g(n),
\]
for all $n,j\in\N$, so this sequence satisfies $\mu(n) \leq Mg(n)$ for all $n\in\N$. By assumption (e), we can see that $\mu(n) = o(n)$. Hence, by Theorem \ref{thm:Helly:Chanturia}, we know that there exists a subsequence $(x_{j_k})_{k\in\N}$ of $(x_j)_{j\in\N}$ pointwise convergent to some function $x\in V[g(n)]$  such that $v(x,n) \leq M g(n).$
\end{proof}

\begin{definition}
If $g\in\mathcal{G}$ and $M=(m_k)_{k\in\N}$ is an increasing sequence of natural numbers, then $g^M:\mathbb{N}\to[0,\infty)$ is the function defined by $g^M(n)=k g(n)$, where $k\in\mathbb{N}$ is given by $n\in[m_k,m_{k+1})$ (and for $n<m_1$ we just put $g^M(n)=g(n)$).
\end{definition}

We are ready to prove the main result of this Section.

\begin{theorem}\label{lem:comp:chant:emb}
Assume we have two functions $g,h\in \mathcal{G}$. The following statements are equivalent:
\begin{enumerate}
\item[(a)] if $K\subseteq V[g(n)]$ is bounded in $V[g(n)]$ and relatively compact in the $\Vert\cdot\Vert_\infty$ norm then $K\subseteq V[h(n)]$ and $K$ is relatively compact in the space $V[h(n)]$,
\item[(b)] $g(n) = o(h(n))$, i.e., 
\[
\lim_{n\to\infty}\frac{g(n)}{h(n)}=0,
\]
\item[(c)] $\Exh(\phi_{g^M})\subseteq\Exh(\phi_h)$ for some increasing sequence $M=(m_k)_{k\in\N}$ of natural numbers.
\end{enumerate}
\end{theorem}
\begin{proof}
(c)$\implies$(b): 
Because $(\frac{n}{g(n)})_{n\in\N}$ is monotonically tending to $\infty$, by dividing $g$ by a constant, without loss of generality we may assume that $g(n)\leq n$ for all $n\in\N$.

Suppose to the contrary that there exists $\alpha>0$ and an increasing sequence $(n_i)_{i\in\mathbb{N}}$ such that 
$$ \forall_{i\in\N} \frac{g(n_i)}{h(n_i)}\geq\alpha.$$

Let $A\subseteq\N$ be such a set that $|A\cap\{1,\ldots,n\}|=\lfloor g(n)\rfloor$ for each $n\in\N$ (note that such $A$ exists: since $(\frac{n}{g(n)})_{n\in\N}$ is nondecreasing, $\frac{n+1}{g(n+1)}\geq \frac{n}{g(n)}$, which implies $g(n+1)\leq g(n)+\frac{g(n)}{n}\leq g(n)+1$, so also $\lfloor g(n+1)\rfloor\leq \lfloor g(n)\rfloor+1$). Then for any $n\geq m_k$ we have
$$\frac{|A\cap\{1,\ldots,n\}|}{g^M(n)}\leq \frac{\lfloor g(n)\rfloor}{kg(n)}\leq \frac{1}{k}, $$
thus $A\in\Exh(\phi_{g^M})$ (by Remark \ref{rem-simpledensity}). On the other hand, there exists $N$ such that for any $n_i\geq N$ we have $g(n_i)\geq 1$. For such $n_i$ we obtain
$$\frac{|A\cap\{1,\ldots,n_i\}|}{h(n_i)}\geq \frac{g(n_i)}{2 h(n_i)}\geq \frac{\alpha}{2}. $$
It follows that $A\not\in\Exh(\phi_{h})$ (see again Remark \ref{rem-simpledensity}), a contradiction with $\Exh(\phi_{g^M})\subseteq\Exh(\phi_h)$.

%
(b)$\implies$(c): Notice that for every $k\in\N$ there exists $m_k$ such that for all $n\geq m_k$ we have
$$\frac{g(n)}{h(n)}\leq \frac{1}{k}. $$
Then, for $n\in[m_k,m_{k+1})$ we obtain
$$ \frac{g^M(n)}{h(n)}=\frac{kg(n)}{h(n)}\leq 1. $$
Therefore, $g^M(n)\leq h(n)$ for all $n\geq m_1$. It is easy to see that this condition implies $\Exh(\phi_{g^M})\subseteq\Exh(\phi_h)$.

(b)$\implies$(a): By Theorem \ref{th:Chant:embed} we can see that $V[g(n)]\subseteq V[h(n)]$. Let us take a bounded set $K\subseteq V[g(n)]$ which is relatively compact in the supremum norm. Let $M>0$ be such that for all $x\in K$ we have $\norm{x}_{\phi_g}\leq M$ and let $(x_n)_{n\in\mathbb{N}}\subseteq K$. By Helly selection theorem (cf. \cite{Chistyakov} and Corollary \ref{cor:Helly:Chanturia} above) we can take the subsequence $(x_{n_k})_{k\in\N}$ pointwise convergent to a function $x\in V[g(n)]$ such that $\norm{x}_{\phi_g}\leq M$. Because $K$ is relatively compact in the uniform convergence norm, we can also assume that $(x_{n_k})_{k\in\N}$ converges to $x$ uniformly. Now we are going to show that 
\[
\lim_{k\to \infty} \Vert x_{n_k}-x\Vert_{\phi_h} = 0.
\]
This will not be a difficult observation: first let us fix $\varepsilon>0$ and take such $n_0\in\N$ that for $n\geq n_0$  
we have $\frac{g(n)}{h(n)} \leq \frac{\varepsilon}{2M}$. Because the subsequence $(x_{n_k}-x)_{k\in\N}$ converges uniformly to $0$ as $k\to+\infty$, we may assume that there exists such $k_0\in \N$ that for any $k\geq k_0$ and for any interval $J\subseteq I$ we have 
\[
|(x_{n_k}-x)(J)| \leq h(1) \varepsilon / n_0.
\]
Then, for any family $J=(J_n)_{n\in\N}$ of nonoverlapping intervals and any $k\geq k_0$ we
may estimate the sum 
\[
\frac{\sum_{i=1}^n |(x_{n_k}-x)(J_i)|}{h(n)}
\]
depending on $n$ by:
\[
\frac{\sum_{i=1}^n |(x_{n_k}-x)(J_i)|}{h(n)} \leq \frac{\sum_{i=1}^{n_0} |(x_{n_k}-x)(J_i)|}{h(n)} \leq n_0 \frac{h(1)  \varepsilon / n_0}{h(n)} \leq \varepsilon,  
\]
for $n\leq n_0$, and by
\[
\frac{\sum_{i=1}^n |(x_{n_k}-x)(J_i)|}{h(n)} \leq \frac{\varepsilon}{2M} \frac{\sum_{i=1}^n |(x_{n_k}-x)(J_i)|}{g(n)} \leq \varepsilon,
\]
for $n>n_0$. This shows that for any $n\in \N$ we have 
\[
\frac{\sum_{i=1}^n |(x_{n_k}-x)(J_i)|}{h(n)} \leq \varepsilon,
\]
what proves the convergence $\Vert x_{n_k}-x\Vert_{\phi_h}\to 0$ as $k\to+\infty$.

(a)$\implies$(b): Contrary to our claim, assume that (a) holds true and that there exists $\alpha>0$ and an infinite increasing sequence $(n_k)_{k\in\N}\subseteq \N$ such that
\[
g(n_k) \geq \alpha h(n_k).
\]
We are going to build a sequence of continuous piecewise linear functions $x_k$ such that $\Vert x_k\Vert_{\phi_g} = 1$, $\Vert x_k\Vert_\infty\to 0$ as $k\to+\infty$, but   $\Vert x_k\Vert_{\phi_h} \geq \varepsilon_0$, so the sequence $(x_k)_{k\in\mathbb{N}}\subseteq V[g(n)]$ does not contain any convergent subsequence.

Let $x_k:I\to\R$ be given as a function such that $x_k(l/n_k) = 0$ for $l\in\{0,...,n_k\}$ being an even number and $x_k(l/n_k) = g(n_k)/n_k$  for $l\in\{0,...,n_k\}$ being an odd number.  For other $t\in I$ the function $g(t)$ is defined as piecewise linear. Then $\Vert x_k\Vert_\infty =  g(n_k)/n_k \to 0$ as $k\to +\infty$. At the same time 
\[
\Vert x_k\Vert_{\phi_h} = \frac{\sum_{j=1}^{n_k} g(n_k)/n_k}{h(n_k)} = \frac{g(n_k)}{h(n_k)}\geq \alpha>0
\]
and $\Vert x_k\Vert_{\phi_g} =1$. This completes the proof.
\end{proof}

\bibliographystyle{amsplain}
\bibliography{compactness-references.bib}

\end{document}